\documentclass[twoside,10.5pt]{article}%
\pdfoutput=1                   
\usepackage{mathrsfs}
\usepackage{pifont}
\usepackage{amsmath}
\usepackage{amsthm}
\usepackage{txfonts}
\usepackage{geometry}
\usepackage{latexsym}
\usepackage{amssymb}
\usepackage{graphicx}
\usepackage{geometry}
\usepackage{xcolor} 
\usepackage{amsmath}
\usepackage{amsfonts}
\usepackage{amssymb}
\usepackage{amsthm}
\usepackage{epic,eepic}
\usepackage{pstricks,pst-3dplot, pstricks-add, pst-node,graphics}\SpecialCoor
\usepackage{verbatim}
\usepackage{graphicx,subfigure}
\usepackage{pst-grad} 
\usepackage{pst-plot} 

\newtheorem{lem}{Lemma}[section]
\newtheorem{rem}{Remark}[section]
\newtheorem{example}{Example}[section]

\makeatletter
\renewcommand\@biblabel[1]{}
\makeatother

\begin{document}
\begin{center}
{\LARGE{Improving Frenet's Frame Using Bishop's Frame}}\\[15pt]
Daniel Carroll$^1$, Emek K\"ose$^1$ \& Ivan Sterling$^1$ 
\end{center}
$^1$ Mathematics and Computer Science Department, St Mary's College of Maryland, St Mary's City, MD, 20686, USA \par
Correspondence: Ivan Sterling,  Mathematics and Computer Science Department, St Mary's College of Maryland, St Mary's City, MD, 20686, USA. Tel: 1-240-431-8185. E-mail: isterling@smcm.edu
\\
\textbf{Abstract}

The main drawback of the Frenet frame is that it is undefined at those points where the curvature is zero.  Furthermore, in the case of planar curves, the Frenet frame does not agree with the standard framing of curves in the plane.   The main drawback of the Bishop frame is that the principle normal vector N is not in it.  Our new frame, which we call the Beta frame, combines, on a large set of curves,  the best aspects of the Bishop frames and the Frenet frames.  It yields a globally defined normal, a globally defined signed curvature, and a globally defined torsion.  For planar curves it agrees with the standard framing of curves in the plane.

\section{Introduction}
Let $\gamma: (a,b) \longrightarrow \mathbb{R}^3$ be a curve in $R^3$.  If $\gamma$ is $C^2$ with $\gamma'(t) \neq 0$, then it carries a Bishop frame $\{T,M_1, M_2\}$.  If $\gamma$ is $C^3$ with $\gamma'(t)$ and $\gamma''(t)$ linearly independent, then it carries a Frenet frame $\{T,N,B\}$.  The main drawback of the Frenet frame is that it is undefined when $\gamma''(t)=0$.  This corresponds precisely to those points where the curvature $\kappa(t)$ is zero.  Also the principle normal vector $N(t)$ of the Frenet frame may have a non-removable discontinuity at these points.   In the case of planar curves, the Frenet frame does not agree with the standard framing of curves in the plane.  Finally the torsion $\tau(t)$ is not defined when $\kappa(t)=0$.  The main drawback of the Bishop frame is that the principle normal vector $N$ is not (except in rare cases) in the set $\{T,M_1, M_2\}$.

The history of the Frenet equations for a curve in $\mathbb{R}^3$ is interesting.  Discovered in 1831 by Senff and Bartels they should probably be called the Senff-Bartels equations.  In 1847 they were rediscovered in the dissertation of Frenet, which was published in 1852.  Independently they were also discovered (and published) by Serret in 1851.  See (Reich, 1973) for details on this early history.

Bishop frames were introduced in 1975 in the Monthly article ``There is More Than One Way to Frame a Curve" (Bishop, 1975).  Bishop frames are now ubiquitous in the literature on curve theory and its applications.

Our new frame, which we call the Beta frame of $\gamma$, combines, on a large set of curves,  the best aspects of the Bishop frames and the Frenet frames.  It yields a globally defined normal $N^\beta$, a globally defined signed curvature $\kappa^\beta$, and a globally defined torsion $\tau^\beta$.  If $\gamma$ is planar, it agrees with the standard framing of curves in the plane.

Our approach was motivated by attempts to improve the details of our work on discrete Frenet frames (Carroll, Hankins, K\"ose \& Sterling).  The Beta frame introduced in this paper discretizes in a natural way consistent with our discrete frame defined in  (Carroll, Hankins, K\"ose \& Sterling).  These ideas are particularly useful in applications, such as DNA analysis and computer graphics.  For example see Hanson's technical report (Hanson, 2007), which discusses several of the issues that we address here.

Notation: $C^0$ means continuous, $C^\infty$ means infinitely differentiable, and $C^\omega$ means real analytic.  For $k \in \mathbb{N}$ we say a function is $C^k$ if its derivatives up to order $k$ are continuous.  If $\gamma'(s) \neq 0$ for all $t$, then $\gamma$ can be reparametrized  by arclength $s$.  If the derivative with respect to $s$ is denoted by $\dot{\gamma}$, then $\Vert \dot{\gamma}(s) \Vert \equiv 1$.  Whenever necessary to simplify notation we assume $0 \in (a,b)$.  Finally, by an abuse of language, we use the term ``planar curve" to mean a curve in $\mathbb{R}^2 \subset \mathbb{R}^3$.

The authors would like to thank the referees for helpful comments.

\section{Before Bishop}
\subsection{The Standard Theory for Curves in $\mathbb{R}^2$}
Let $\gamma: (a,b) \longrightarrow \mathbb{R}^2$ be a $C^2$ curve in two-space with $\Vert \dot{\gamma}(s) \Vert \equiv 1$.  $T(s):=\dot{\gamma}(s)$ is called the unit tangent vector to $\gamma$ at $s$.  Let the normal $N(s)$ be the unique unit vector orthogonal to $T(s)$ such that $\{T(s),N(s)\}$ is positively oriented.  (We think of this $N$ as the good normal, as opposed to the bad normal below.)  The signed curvature $\kappa^{signed}$ is defined by 
\[ \dot{T}(s) =: \kappa^{signed}(s) N(s).\]
The ``unsigned curvature", the curvature of the osculating circle, is 
\[\kappa(s) :=|\kappa^{signed}(s)|.\]

Alternatively, not wisely, one could first define curvature $\kappa^{bad}$ by
\[\kappa^{bad}(s) := \Vert \dot{T}(s) \Vert\]
and then define, when $\kappa^{bad}(s) \neq 0$, the normal $N^{bad}(s)$ by
\[N^{bad}(s) := \frac{\dot{T}(s)}{\Vert \dot{T}(s) \Vert} = \frac{\dot{T}(s)}{\kappa^{bad}(s)}.\]
There seems to be little advantage to this ``bad" alternative and at least two drawbacks.  The first drawback is that $N^{bad}(s)$ is not defined when $\kappa^{bad}(s)=0$, hence $N^{bad}(s)$ has at best a removable discontinuity when $\kappa^{bad}(s)=0$.  The second drawback is that if $\gamma(s)$ changes concavity then $N^{bad}(s)$ has a jump discontinuity as for example in Figure \ref{frenetbetax3}.

\begin{figure}[h]
\begin{center}
\includegraphics[scale=.6]{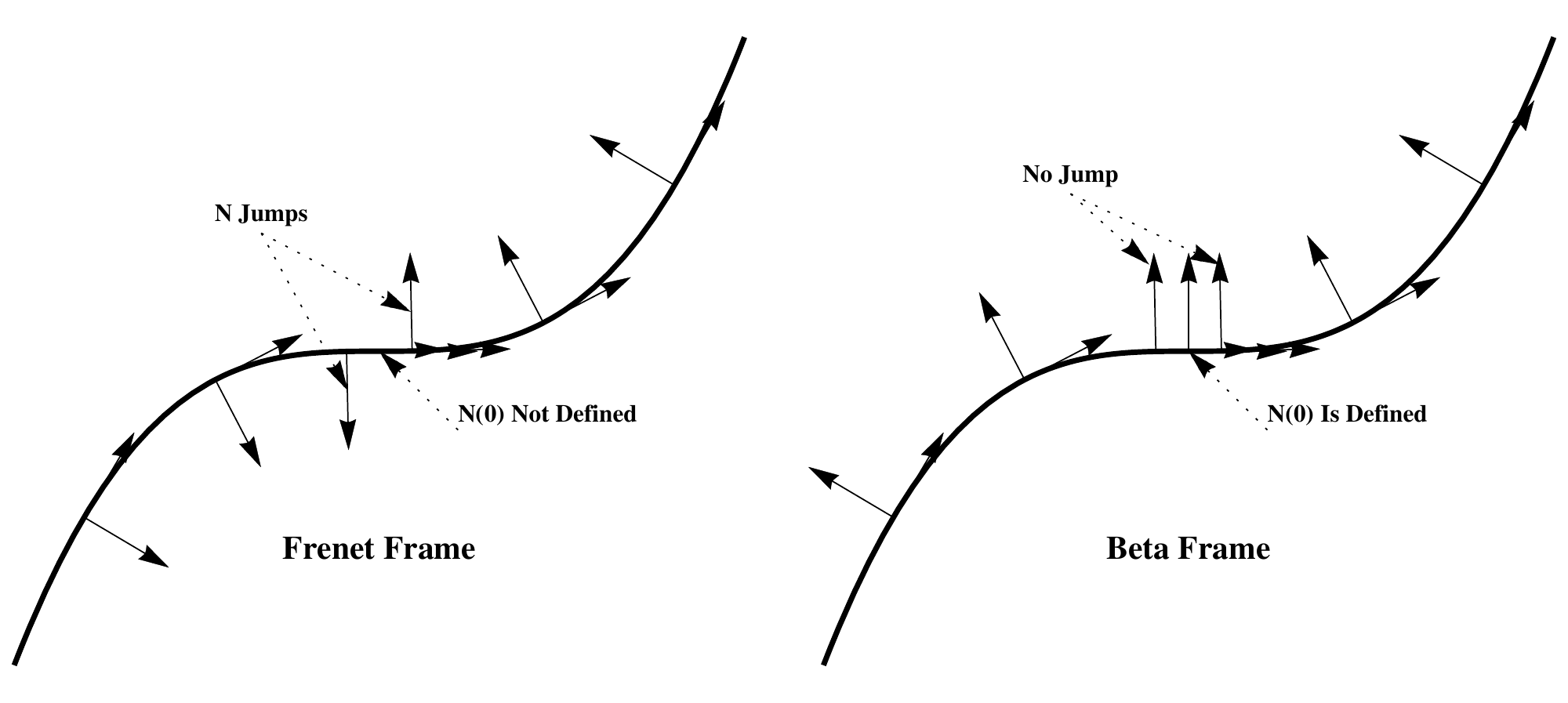}
\caption{The Difference Between the Frenet Frame and the Beta Frame}\label{frenetbetax3}
\end{center}
\end{figure}

\subsection{The Standard Theory for Curves in $\mathbb{R}^3$}
Even though the alternative is bad, it is this bad alternative which is used in the standard Frenet framing for $C^3$ curves $\gamma: (a,b) \longrightarrow \mathbb{R}^3$ with $\Vert \dot{\gamma}(s)\Vert \equiv 1$.  We first define Frenet's curvature $\kappa^f$ by
\[\kappa^f(s) := \Vert \dot{T}(s) \Vert\]
and then define, when $\kappa^f(s) \neq 0$, the Frenet (principal) normal $N^f$ by
\[N^f(s) := \frac{\dot{T}(s)}{\Vert \dot{T}(s) \Vert} = \frac{\dot{T}(s)}{\kappa^f(s)}.\]
Note that for planar curves $\kappa^f = \kappa^{bad}$ and  $N^f = N^{bad}$.  As mentioned above, for planar curves the Frenet frame may not agree with the positively oriented standard frame for curves in $\mathbb{R}^2$.  Roughly speaking the purpose of this paper is to do away with ``bad"  (or non-existent) normals whenever possible.
The Frenet (principal) binormal $B^f$ is defined by $B^f(s) = T(s) \times N^f(s)$.

If $\gamma$ is $C^3$ with $\Vert \dot{\gamma}(s)\Vert \equiv 1$ and $\kappa^f(s) \neq 0$, then $\tau^f(s)$ is defined by 
\[\tau^f = \frac{\langle \dot{\gamma} \times \ddot{\gamma}, \dddot{\gamma}\rangle}
{{\kappa^f}^2}.\]
The torsion $\tau^f(s)$ measures the rate of change of the osculating plane, the plane spanned by $T(s)$ and $N^f(s)$.
One has the Frenet equations:
\begin{alignat*}{4}
\dot{T} &=& \kappa^f N^f, \\
\dot{N^f} &=-\kappa^f T &&+\tau^f B^f, \\
\dot{B^f} &=&-\tau N^f.
\end{alignat*}
The set, $\{T,N^f,B^f\}$, is called the Frenet frame of $\gamma$.

\section{Discussing the Problem}
In Chapter 1 of (Spivak, 1990) Spivak discusses why we cannot obtain a signed curvature $\kappa^{signed}$ for curves in three space and why we cannot, in general, hope to define torsion $\tau(s_0)$ at points where $\kappa^f(s_0)=0$.  

With respect to signed curvature, Spivak points out that there is no natural way to pick a vector orthogonal to a given $T(s_0)$ in $\mathbb{R}^3$.  Furthermore the Frenet frame, in particular the principle normal $N^f(s)$, is only defined on intervals where $\kappa^f(s) \neq 0$.  There may be no consistent way to choose the normal after passing through a point $s_0$ with $\kappa^f(s_0)=0$.  If $\kappa^f(s_0) = 0$, $N^f(s)$ may have a non-removable discontinuity at $s_0$.  However, we are able, for a large set of curves, to define a new frame, the Beta frame, $\{T,N^\beta, B^\beta\}$, which is at least $C^0$ and is defined even when $\kappa^f(s) =0$.  Furthermore $N^\beta(s)= \pm N^f(s)$, whenever $N^f(s)$ is defined.  Once we have a global definition of $N^\beta$, we define the signed curvature $\kappa^\beta$ ``the good way" by
\[\dot{T} =: \kappa^\beta N^\beta.\]
For all $s$ we will have $\kappa^\beta(s) = \pm \kappa^f(s)$.  If $\gamma$ is planar, then $N^\beta = N$ (the good normal) and $\kappa^\beta = \kappa^{signed}$.

With respect to torsion, Spivak argues that one cannot define $\tau^f(s)$ when $\kappa^f(s)=0$, because there exist examples where no reasonable definition would make sense. 
\begin{example}\label{spivak}(Spivak's Example.  See Figure \ref{spivakcurv}.) If $\gamma$ is the $C^\infty$ curve defined by
\[ \gamma(t) := 
\left\lbrace \begin{array}{cl} (s,e^{1/s^2},0) & \mbox{if}\; t>0,  \\
(0,0,0) & \mbox{if}\;\  t=0,\\
(s,0,e^{1/s^2}) & \mbox{if}\;\ t<0.
\end{array} \right. \]
then $\tau(t)=0$ everywhere except $t=0$.  But at $t=0$ the osculating plane jumps by an angle $\frac{\pi}{2}$.  Any attempt to define $\tau(0)$ would involve distributions and delta functions, which we will not pursue in this paper.
\end{example}

\begin{figure}[h!]
\begin{center}
\includegraphics[scale=.6]{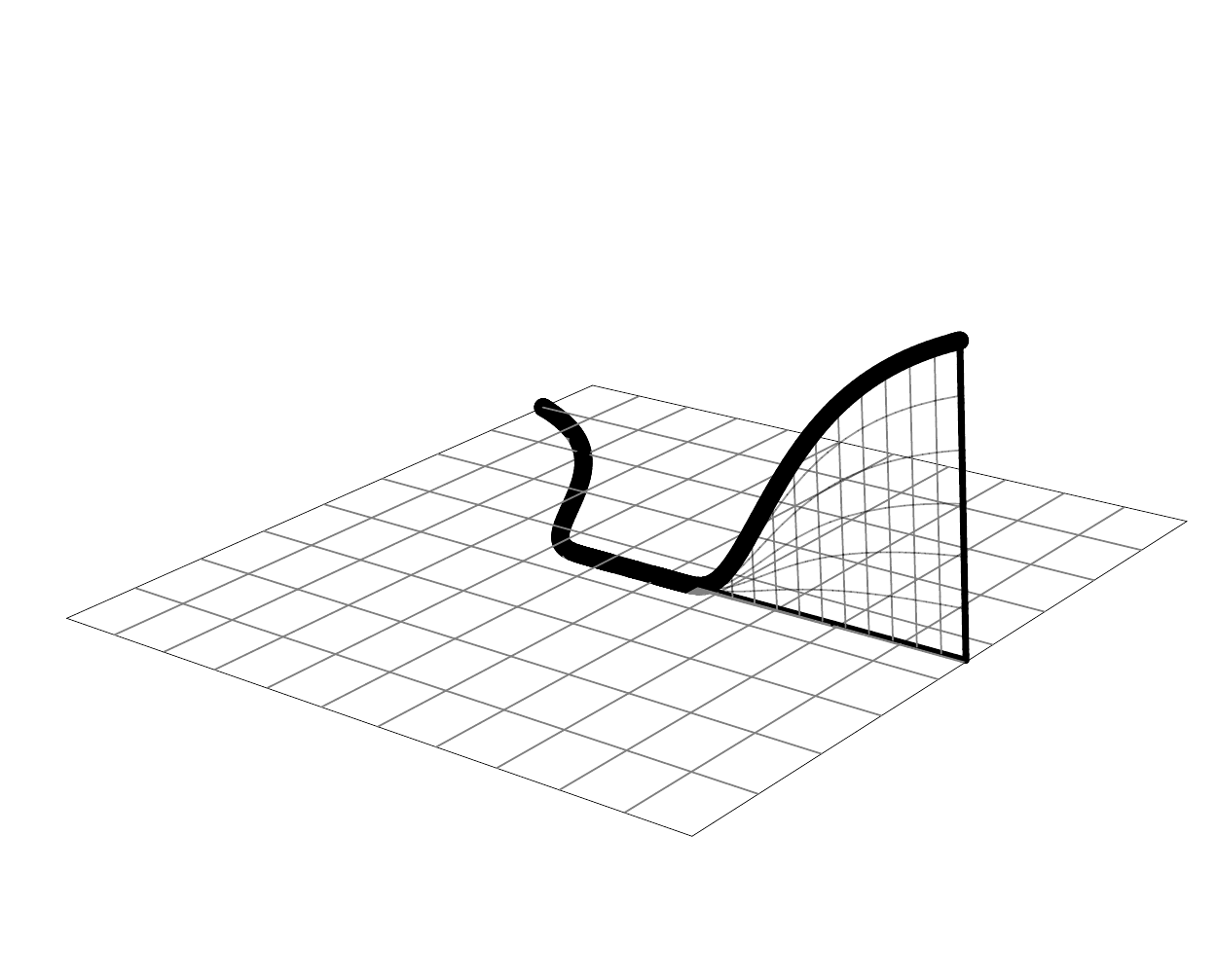}
\caption{Spivak's Example}\label{spivakcurv}
\end{center}
\end{figure}

We take a different approach than Spivak.  Instead of focusing on the examples where no reasonable definition is possible, we give conditions, satisfied by many curves of interest, that allow $\tau^\beta(s_0)$ to be defined even at points $s_0$ where $\kappa^f(s_0)=0$.  Furthermore $\tau^\beta = \tau^f$ whenever $\tau^f$ is defined.  Finally we will still have the Frenet equations:
\begin{alignat*}{4}
\dot{T} &=& \kappa^\beta N^\beta, \\
\dot{N^\beta} &=-\kappa^\beta T &&+\tau^\beta B^\beta, \\
\dot{B^\beta} &=&-\tau^\beta N^\beta.
\end{alignat*}

\section{The Bishop Frame}
Let $\gamma: (a,b) \longrightarrow \mathbb{R}^3$ be a $C^2$ with $\Vert \dot{\gamma} \Vert \equiv 1$.  The construction of a Bishop frame \cite{B} is based on the idea of relatively parallel fields.  In particular, a normal vector field $M(s)$ along a curve is called relatively parallel if 
\[\dot{M}(s) = g(s) T(s) \]
for some function $g(s)$.  A single unit normal vector $M_0$ at $\gamma(s_0)$ generates a unique relatively parallel unit normal vector field $M(s)$ along $\gamma$ with $M(s_0)=M_0$.  Moreover, any orthonormal basis $\{T, M_{1_0},M_{2_0}\}$ at $\gamma(s_0)$ generates a unique $C^1$ orthonormal frame $\{T,M_1,M_2\}$.  Bishop's equations are similar to the Frenet equations:
\begin{alignat*}{4}
\dot{T} & = & \kappa_1 M_1 &+\kappa_2 M_2, \\
\dot{M_1} & =  -\kappa_1 T, &&\\
\dot{M_2} & =  -\kappa_2 T. &&
\end{alignat*}
We have $\kappa^f(s)= \sqrt{\kappa_1^2(s) + \kappa_2^2(s)}$.  If $\kappa^f(s) \neq 0$, then 
\[N^f(s) = \frac{\kappa_1(s)}{\kappa^f(s)} M_1(s) + \frac{\kappa_2(s)}{\kappa^f(s)} M_2(s).\]
On sub-intervals of $(a,b)$ where $\kappa^f(s) \neq 0$, there exists a $C^0$ function $\theta(s)$ such that $N^f(s)=\cos \theta(s) M_1(s) + \sin \theta(s) M_2(s)$.  If moreover $\gamma$ is $C^3$, then $\tau^f(s) = \dot{\theta}(s)$.  In most applications the normal portion of the Bishop frame, $span\{M_1,M_2\}$, is usually written using this polar coordinate approach in complex form:
\[(\kappa_1,\kappa_2)= \kappa e^{i \int{\tau}}.\]
We will investigate these polar coordinates in some detail, but not using the complex form.  

\begin{rem} The function $\int \kappa$ is called the turn of $\gamma$.  We have $\theta = \int \tau$. 
$\theta$ is related to the twist and the writhe of $\gamma$ which we won't discuss here.
\end{rem}

\section{Bishop's Normal Development Curve} \label{normdev}
If $\gamma$ is $C^2$ and $\Vert \dot{\gamma}(s) \Vert \equiv 1$, then 
the curve $(\kappa_1(s),\kappa_2(s))=(r(s),\theta(s))$ in $\{M_1, M_2\}$ space is called Bishop's normal development of $\gamma$.  The normal development of $\gamma$ is determined up to rotation by a constant angle in the $\{M_1,M_2\}$ plane and a curve $\gamma$ is determined up to congruence by its normal development.  The parameter $s$ is an arclength parameter for $\gamma$, but in general is  {\it not} an arclength parameter for the normal development of $\gamma$.  The normal development of a line is the constant curve whose image is the origin and the normal development of a circle is the constant curve whose image is $(\kappa^f_{const} \neq 0,\theta_{const})$.  

If $\gamma$ is planar, then it has vanishing torsion and the normal development is given by $(r(s),\theta_{const})$.  For any $\gamma \in \mathbb{R}^3$, zeros of the normal development corresponds to points of zero curvature on $\gamma$.  If the normal development $(r(s),\theta(s))$ approaches the origin along a line and leaves the origin along a different line, the corresponding $\gamma$ is like Spivak's Example \ref{spivak}  above, it jumps from lying in one plane to lying in a different plane.  The normal development of a helix is a constant speed circle around the origin.  A $C^2$ curve is spherical if and only if its normal development lies on a line not through the origin \cite{B}. 

If $\gamma$ is $C^3$, we have seen that when $\kappa^f(s) \neq 0$, we have, for some function $\theta(s)$, $\tau^f(s) = \dot{\theta}(s)$.  In particular,  $\tau^f(s)$ will change signs at the local extrema of $\theta(s)$.   Curves of constant torsion $\pm 1$ correspond to $\theta(s)=\pm s + \theta_0$ with $r(s)$ arbitrary.  Curves of constant curvature $1$ correspond to curves with $r(s) \equiv 1$ and $\theta(s)$ arbitrary.

\section{Curves in Polar Coordinates Through (0,0)}{\label{lift}}
\subsection{Polar Lifts}
As we have seen points where $\kappa^f(s)=0$ on $\gamma$ correspond to the points on the normal development with $(r(s),\theta(s))=(0,0)$.  If $\kappa^f(s) \equiv 0$ on an interval, then $\gamma$ is a line segment on that interval, and the normal development remains at $(0,0)$ on that interval.  Dealing with the case of ``piecewise" defined curves including line segments is a delicate problem which we will address elsewhere.  We will only consider curves that have isolated points of zero curvature.  We assume that $(\kappa_1(s),\kappa_2(s))$ has an isolated zero at $s_0$.

Recall that $(\kappa_1(s),\kappa_2(s))$ is $C^0$.  Nevertheless there may not exist any pair of $C^0$ functions $\tilde{r}(s)$, $\tilde{\theta}(s)$ such that $(\tilde{r}(s),\tilde{\theta}(s)) = (\kappa_1(s),\kappa_2(s))$.  Even if $(\kappa_1(s),\kappa_2(s))$ is $C^\infty$, there may not exist such $C^0$ functions $\tilde{r}(s)$, $\tilde{\theta}(s)$.  More precisely, let the Cartesian plane be defined by
\[\mathbb{R}^2_{(x,y)} = \{(x,y) | -\infty < x < \infty, -\infty < y < \infty\},\]
and let the (extended) Polar plane be defined by
\[\mathbb{R}^2_{(\tilde{r},\tilde{\theta})} = \{(\tilde{r},\tilde{\theta)} | -\infty < \tilde{r} < \infty, -\infty < \tilde{\theta} < \infty\}.\]
Note that usually the definition of Polar plane restricts to $r >0$ and $\theta \in [0,2 \pi)$, but we will use ``Polar plane" in the extended sense as defined in $\mathbb{R}^2_{(\tilde{r},\tilde{\theta})}$.
Let $\pi: \mathbb{R}^2_{(r,\theta)} \longrightarrow \mathbb{R}^2_{(x,y)}$ be given by $\pi(r,\theta) = (r \cos \theta, r \sin \theta)$.  Note that $\pi$ is $C^\omega$.  Curves which are $C^0$ in the Polar plane project down to curves which are $C^0$ in the Cartesian plane.  However not all $C^0$ curves in the Cartesian plane are projections of $C^0$ curves in the Polar plane.  Examples include those Spivak-like curves which enter the origin from one direction $\theta_1$ and leave from a ``non-parallel" direction $\theta_2$.  Figure \ref{nolift} shows this and two other cases.

\begin{figure}[h!]
     \begin{center}
        \subfigure[Spivak-Like Lift]{
            \label{fig:spivaklift}
            \includegraphics[width=0.3\textwidth]{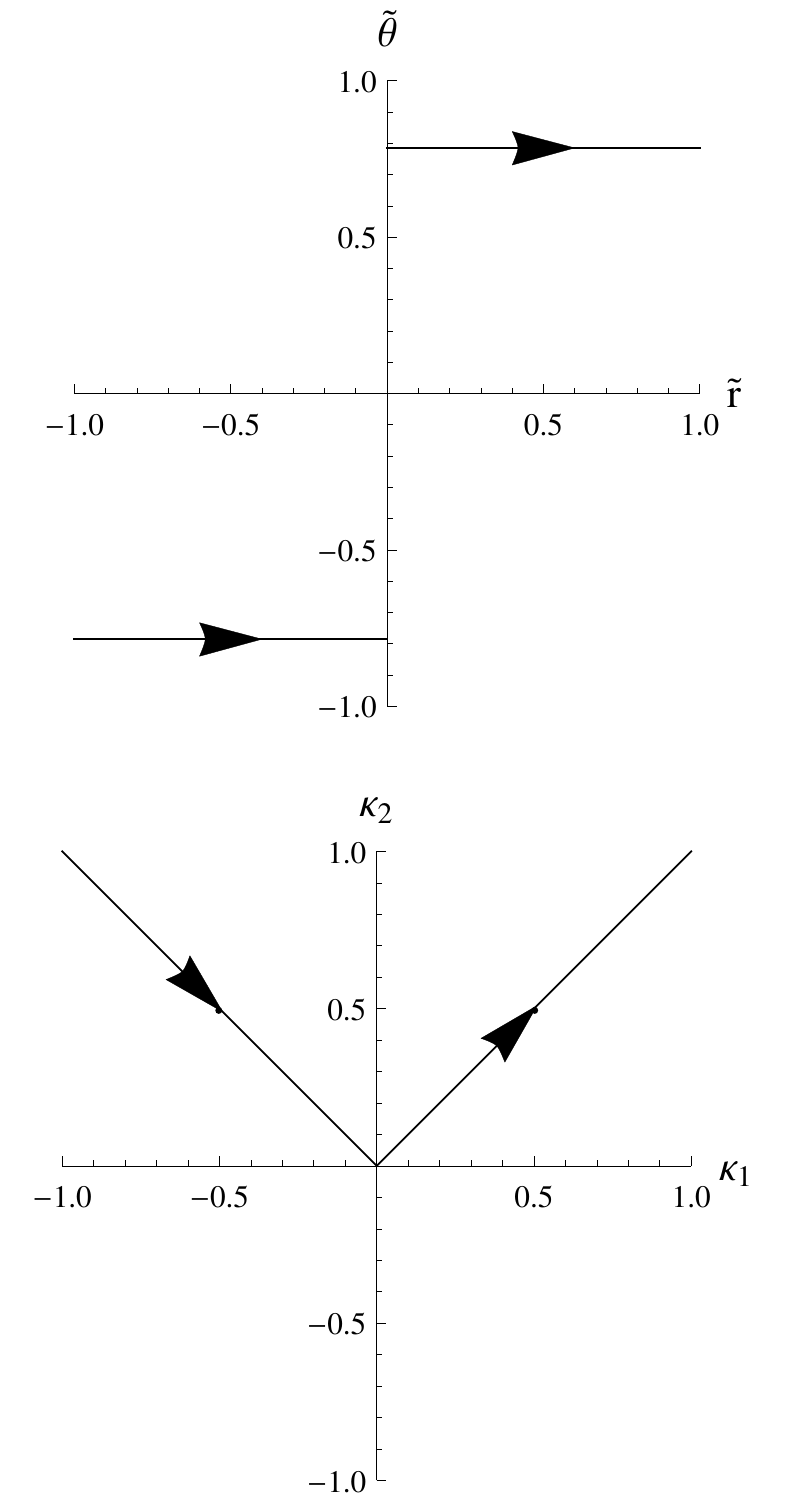}
        }
        \subfigure[$(\tilde{r},\tilde{\theta})=(s, \frac{\pi}{4} + \frac{\pi}{8} \sin{\frac{1}{s}})$]{
           \label{fig:noliftoscil}
           \includegraphics[width=0.3\textwidth]{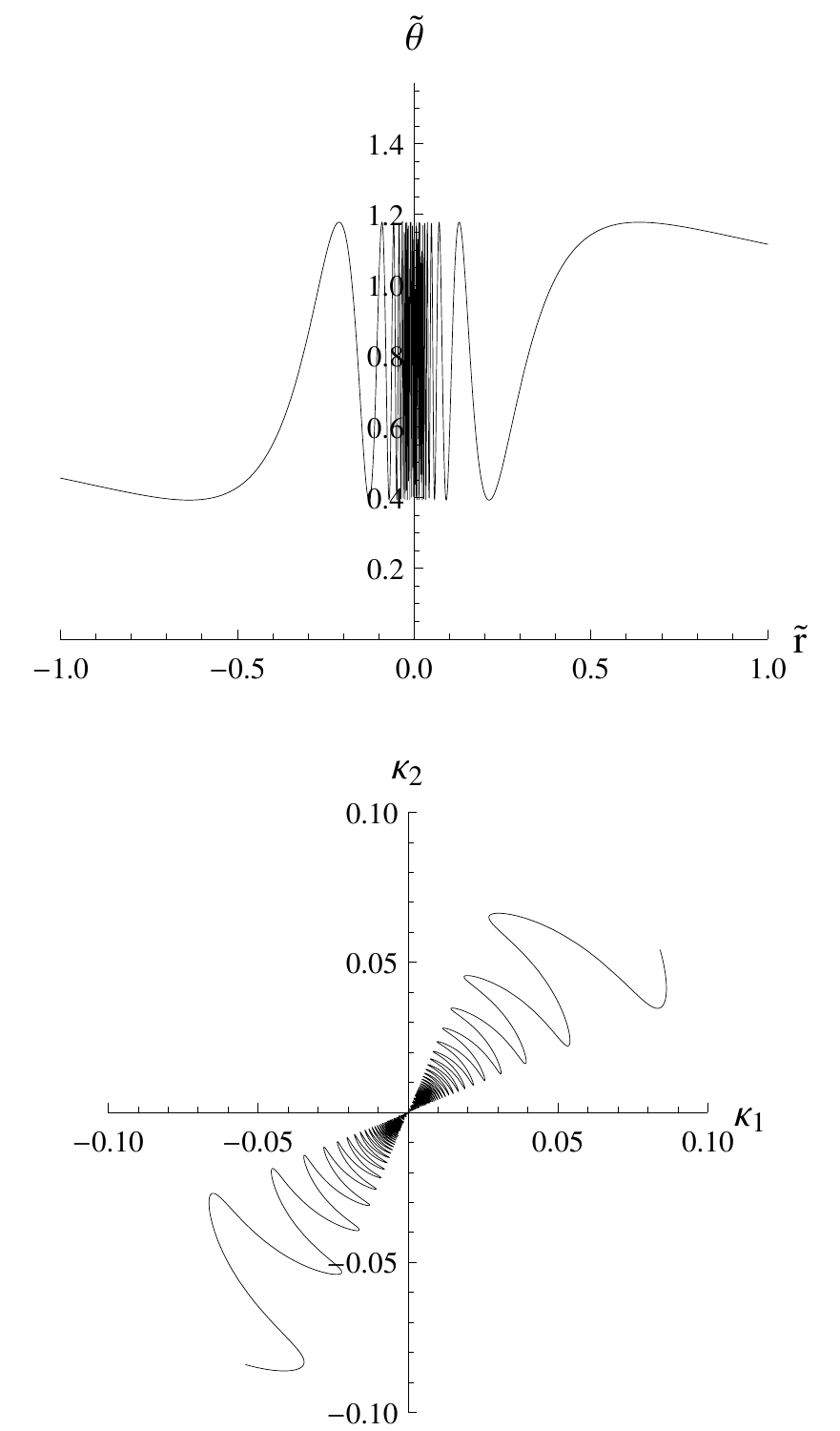}
        } 
        \subfigure[$(\tilde{r},\tilde{\theta})=(s,\frac{1}{\sqrt{s}})$]{
           \label{fig:noliftspiral}
           \includegraphics[width=0.3\textwidth]{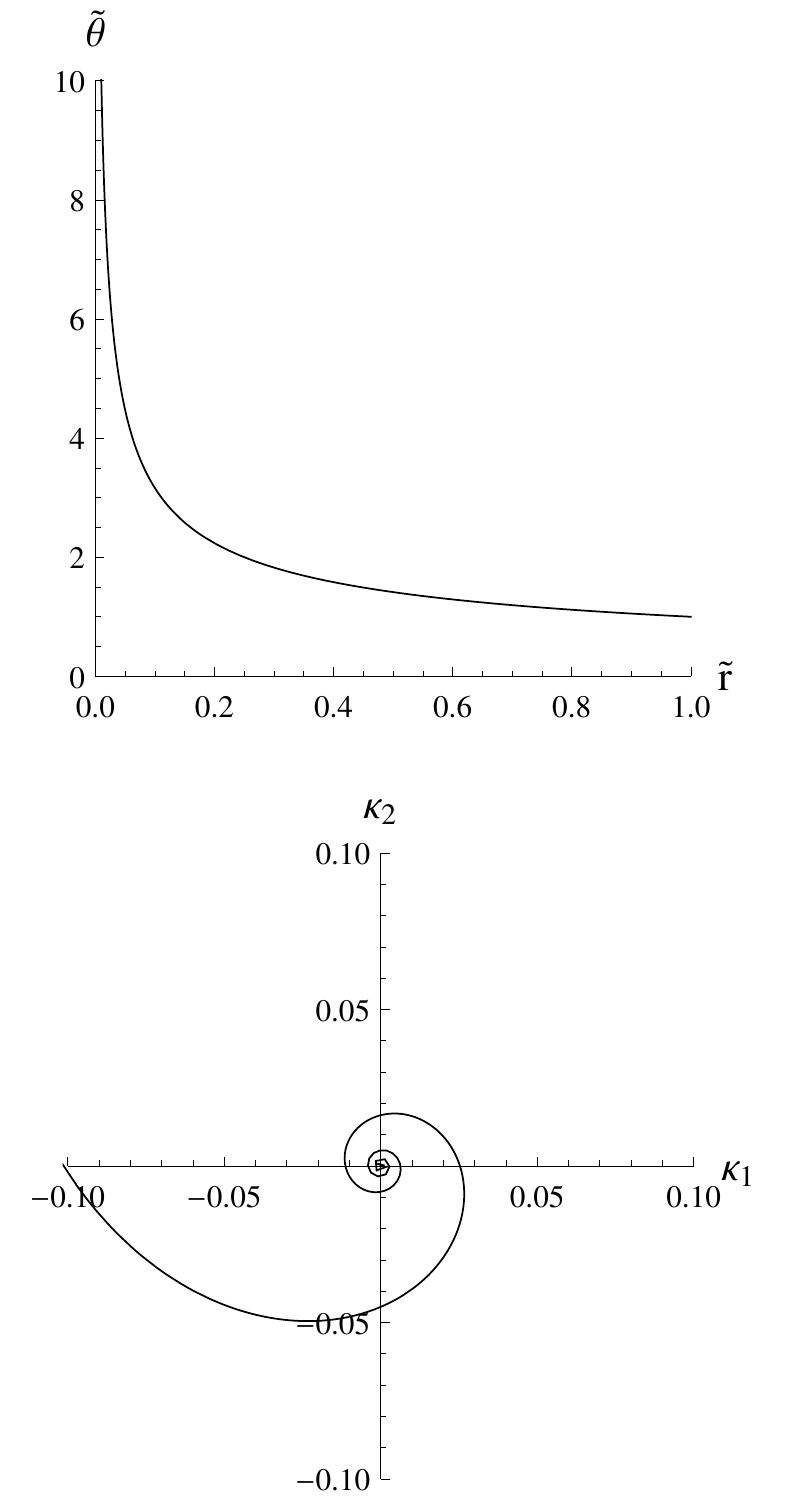}
        }
    \end{center}
    \caption{Some Examples Where $\tilde{\theta}$ Is Not Continuous}  \label{nolift}
\end{figure} 

Specifically we ask, given $(\kappa_1,\kappa_2): (a,b) \longrightarrow \mathbb{R}^2$ is $C^0$, under what conditions does there exist a $C^0$ lift $(\tilde{r},\tilde{\theta}): (a,b) \longrightarrow \mathbb{R}^2$ such that $\pi \circ (\tilde{r}(s),\tilde{\theta}(s)) = (\kappa_1(s), \kappa_2(s))$?  If $(\kappa_1,\kappa_2)$ is never $(0,0)$ then finding a lift $(\tilde{r}(s),\tilde{\theta}(s))$ is trivial.  If $(\kappa_1(s_0),\kappa_2(s_0)) = (0,0)$ is an isolated zero, then there exist a neighborhood $D$ of $s_0$ in $(a,b)$ such that $(\kappa_1(s),\kappa_2(s)) \neq (0,0)$ except at $s_0$.  At each such isolated zero we require the existence of a corresponding $C^0$ lift $(\tilde{r},\tilde{\theta}): D \longrightarrow \mathbb{R}_{(\tilde{r},\tilde{\theta})}^2$.  The first thought that comes at the reader's mind may be that the lift exists if and only if the derivatives in $s_0^+$ and $s_0^-$ coincide.  However, Figure \ref{yeslift}(b), is a simply example to show this is not true.  In fact much more subtle examples can be constructed.  Figures \ref{yeslift} and \ref{liftoscil} show some simple cases where $C^0$ lifts exist.  A unique $C^0$ global lift $(\tilde{r},\tilde{\theta})$ is then constructed by patching together the pieces through $(0,0)$ and pieces which avoid $(0,0)$.  

At this point we'd like to emphasize that we cannot ``just take the Bishop frame".  The problem is that although the Bishop frame is indeed global and $C^0$, the relationship with the Frenet frame (and in particular the principal normal) is not.  To recover a $C^0$ ``principal normal" a detailed analysis of these zeros is required.  Precisely the lack of such a global normal in the literature was the primary motivation of this paper.

\begin{figure}[h!]
     \begin{center}
        \subfigure[$(\tilde{r},\tilde{\theta})=(s,s + \frac{\pi}{4})$]{
            \label{fig:simpleL}
            \includegraphics[width=0.3\textwidth]{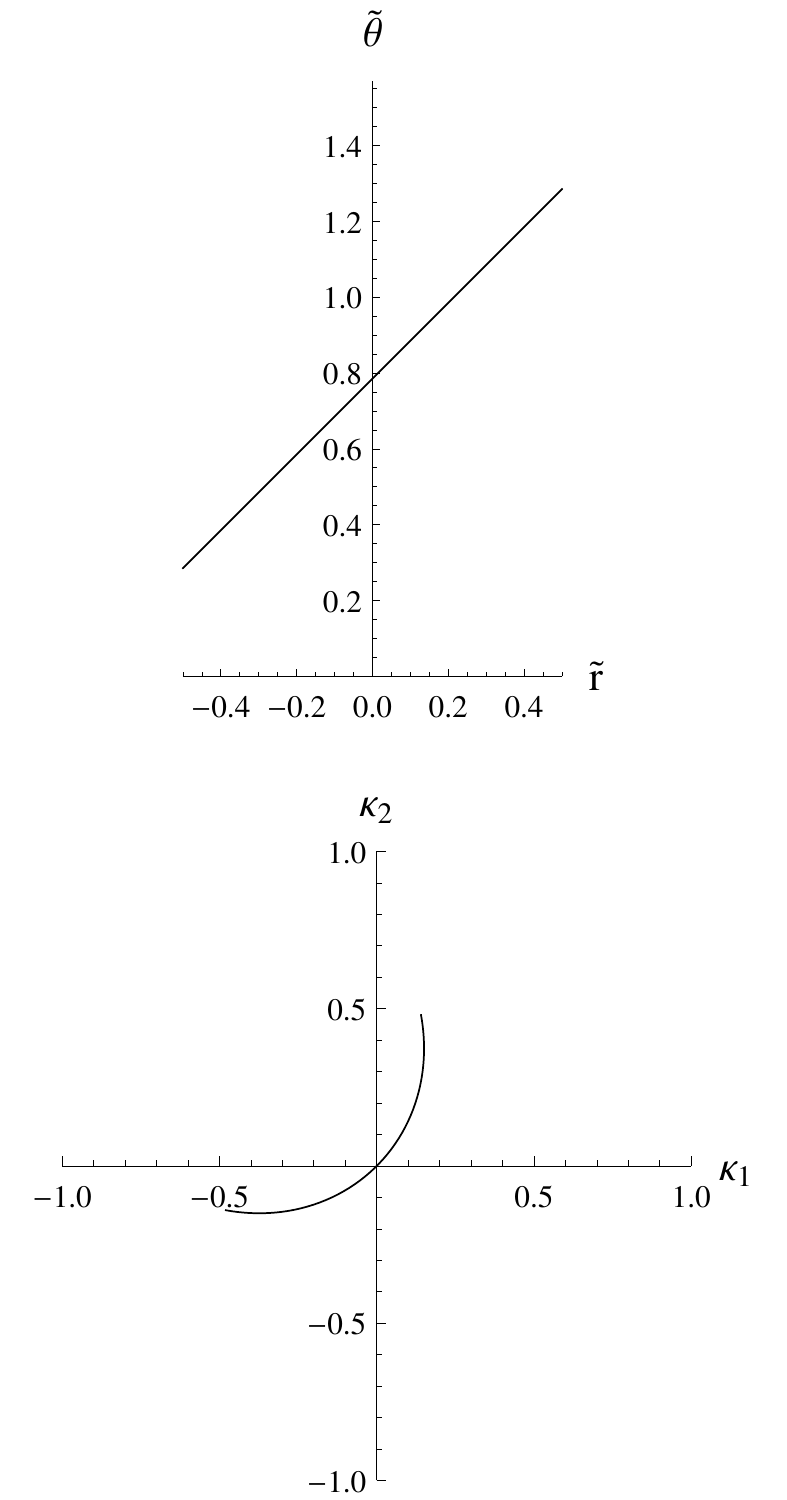}
        }
        \subfigure[$(\tilde{r},\tilde{\theta})=(|s|,s + \frac{\pi}{4})$]{
           \label{fig:simpleK}
           \includegraphics[width=0.3\textwidth]{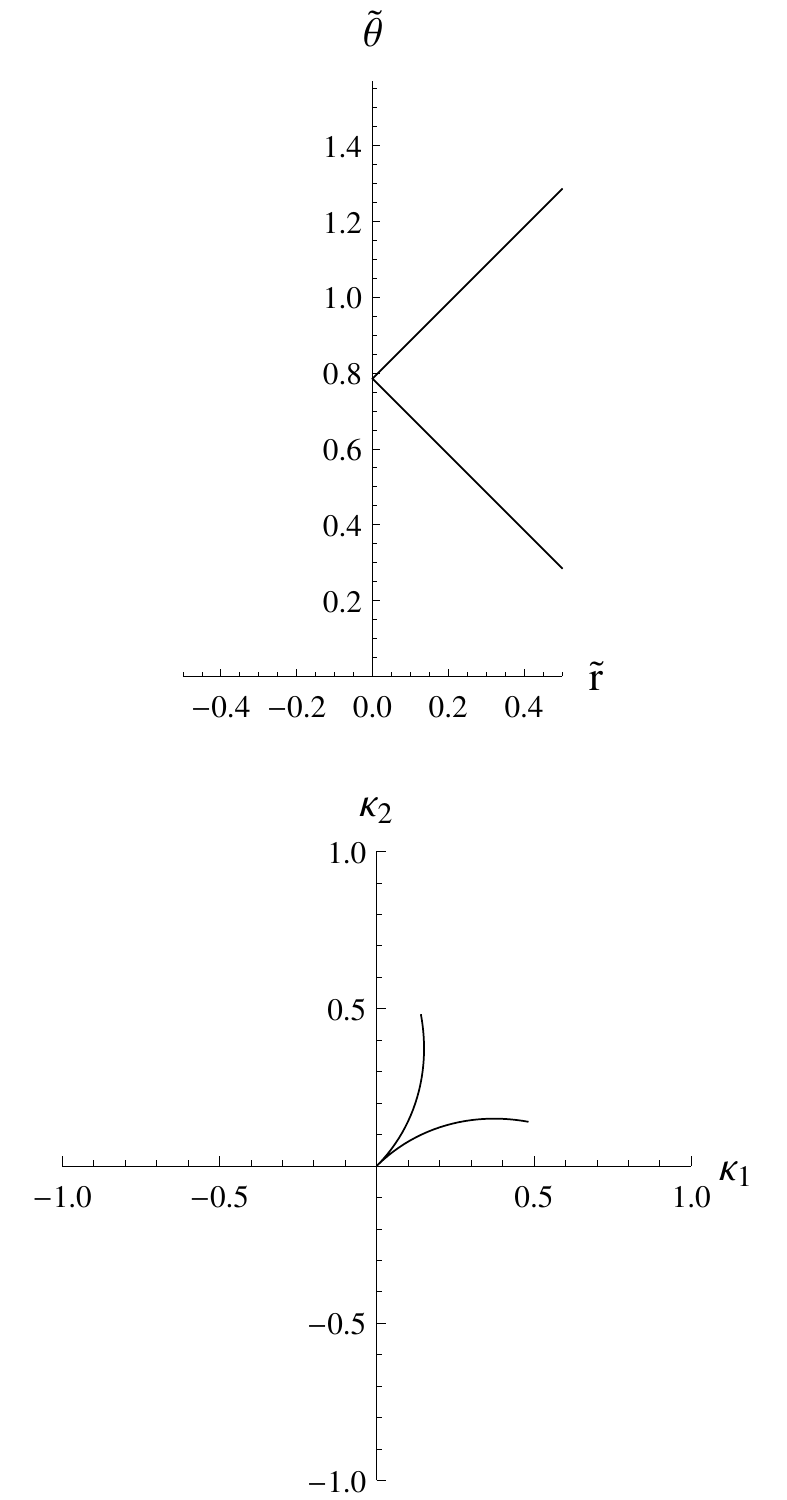}
        } 
        \subfigure[$(\tilde{r},\tilde{\theta})=(s,|s| + \frac{\pi}{4})$]{
           \label{fig:simpleV}
           \includegraphics[width=0.3\textwidth]{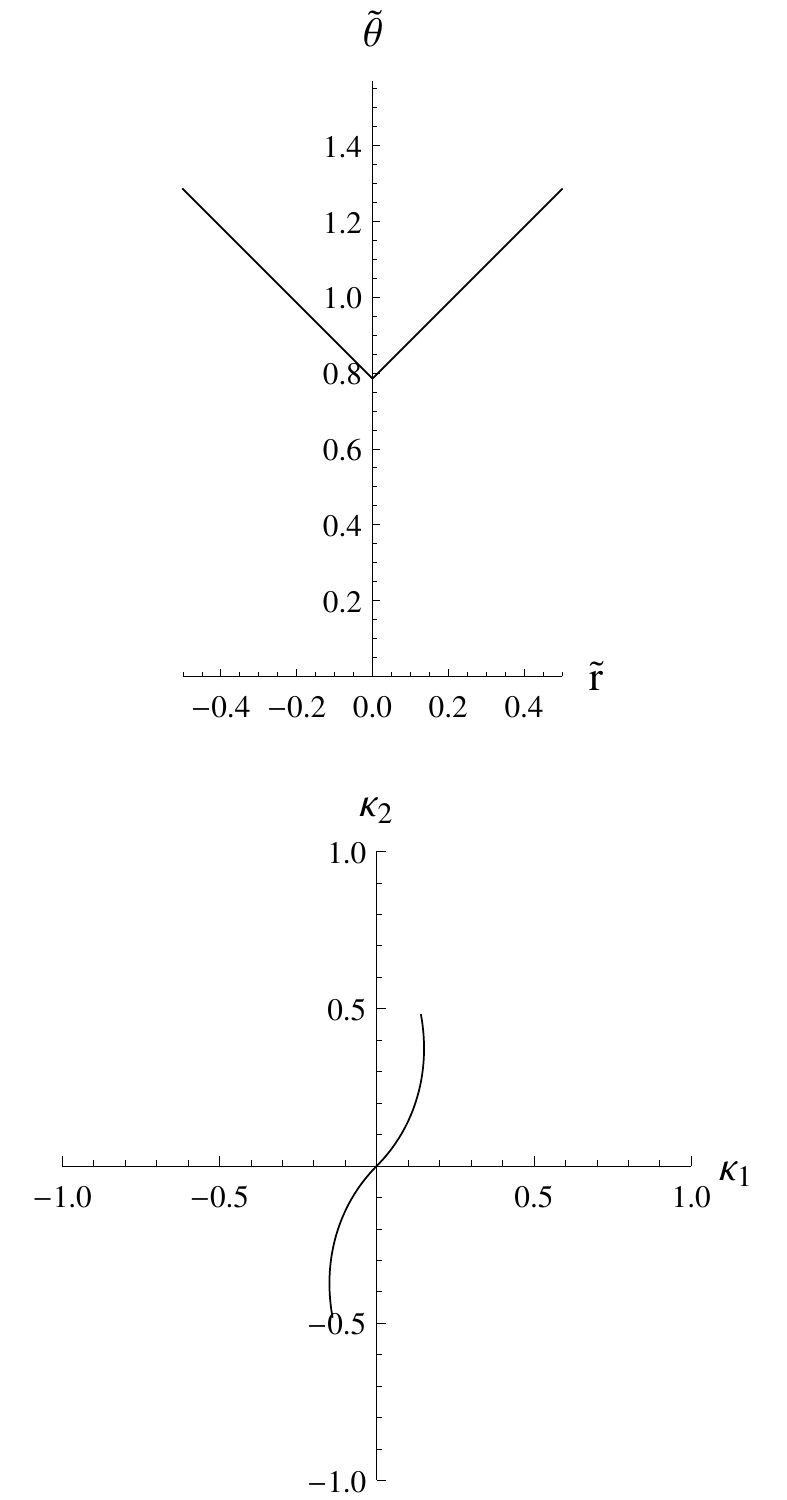}
        }
    \end{center}
    \caption{Simple Examples Where $\tilde{\theta}$ Is Continuous}  \label{yeslift}
\end{figure} 

\begin{figure}[h!]
\begin{center}
\includegraphics[scale=.6]{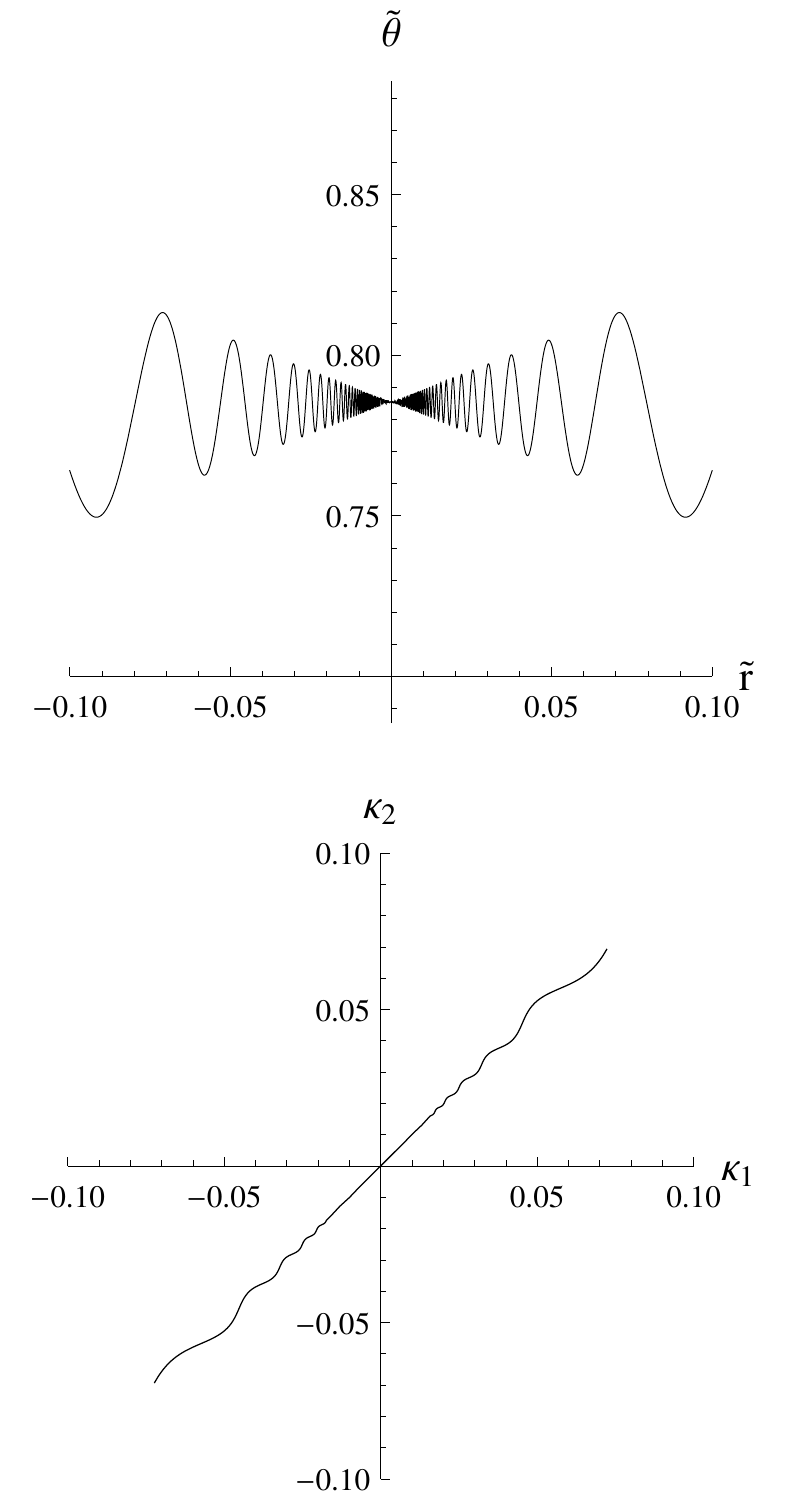}
\caption{$(\tilde{r},\tilde{\theta})=(s,\frac{\pi}{4} + s \frac{\pi}{8} \sin{\frac{1}{s}})$ Another Example Where $\tilde{\theta}$ Is Continuous}\label{liftoscil}
\end{center}
\end{figure}

\subsection{Main Lemma}
This subsection and the following Lemma are technical and unappealing, but straightforward.

We define $\hat{\theta}: \mathbb{R}^2\! - \!(0,0) \longrightarrow \left(-\frac{\pi}{2},\frac{\pi}{2} \right]$ by
\[ \hat{\theta}(x,y) := 
\left\lbrace \begin{array}{cl} \tan^{-1}\left( \frac{y}{x}\right)  & \mbox{if}\; x \neq 0, \\
\frac{\pi}{2} & \mbox{if}\; x=0, y \neq 0, \\
\mbox{undefined} & \mbox{if}\; x=0, y=0.
\end{array} \right. \]
\noindent If $\phi = \mbox{Arg}(x+iy)$ is defined to be the unique argument of $x+iy$ in $[0,2\pi)$, then
\[ \hat{\theta} = 
\left\lbrace \begin{array}{cl} \phi  & \mbox{if}\; \phi \in \left[0,\frac{\pi}{2}\right], \\
\phi - \pi  & \mbox{if}\; \phi \in \left(\frac{\pi}{2}, \frac{3 \pi}{2}\right], \\
\phi- 2\pi & \mbox{if}\; \phi \in \left(\frac{3 \pi}{2},2 \pi \right).
\end{array} \right. \]
A graph of $\hat{\theta}$ as a function of $\phi$ is shown in Figure \ref{canada2}.
\begin{figure}[h!]
\begin{center}
\includegraphics[scale=.6]{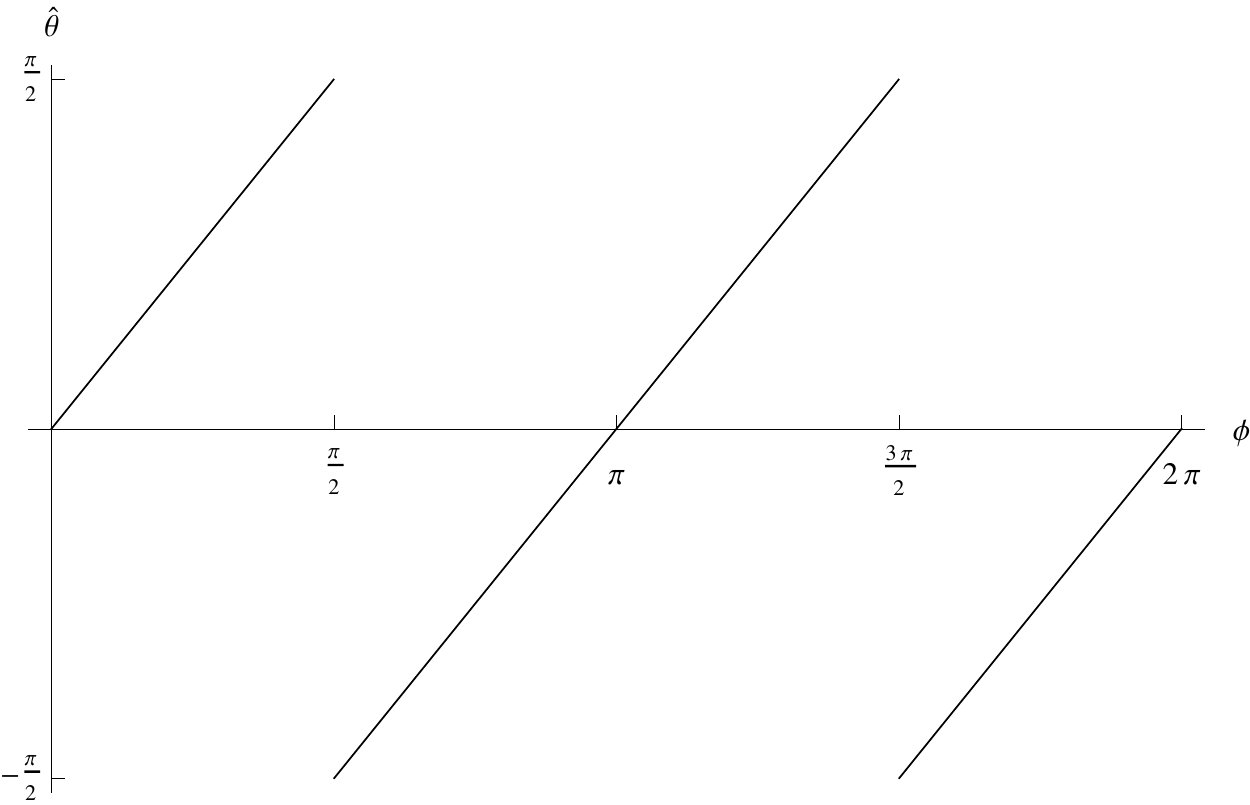}
\caption{$\hat{\theta}$ as a function of $\phi$}\label{canada2}
\end{center}
\end{figure}
Note
\[ \hat{\theta}\left(\pi (\tilde{r}(s),\tilde{\theta}(s))\right) = 
\left\lbrace \begin{array}{cl} \tilde{\theta}(s) \!\!\!\!\!\mod \!\pi  & \mbox{if}\; 0 \leq \tilde{\theta}(s) \!\!\!\!\!\mod \!\pi \leq \frac{\pi}{2}, \\
\left(\tilde{\theta}(s) \!\!\!\!\!\mod \!\pi\right)-\pi  & \mbox{if}\; \frac{\pi}{2} < \tilde{\theta}(s) \!\!\!\!\!\mod \!\pi < \pi.
\end{array} \right. \]
\noindent $\tan^{-1}$ and hence $\hat{\theta}$ jumps at $\frac{\pi}{2}$.  Thus special care must be taken to deal with the case of $(\kappa_1,\kappa_2)$ oscillating across the $y$-axis ($x \equiv 0$) infinitely often as it approaches the origin.  This would cause $\hat{\theta}$ to oscillate wildly between near $\frac{\pi}{2}$ and near $-\frac{\pi}{2}$.  To avoid this cosmetic problem we (in this case only) rotate $(\kappa_1,\kappa_2)$ by an angle $\frac{\pi}{4}$ (so the curve is now oscillating harmlessly about $-\frac{\pi}{4}$).  Then we can check its behavior without the jumping.  We will refer to this special case as Case 3; Case 1 being $-\frac{\pi}{2} < \hat{\theta}^+ < \frac{\pi}{2}$ and Case 2 being $\hat{\theta}^+ = \frac{\pi}{2}$.

Let 
\[ \hat{\theta}^+ := \lim_{s \rightarrow s_0^+} \hat{\theta}(\kappa_1(s), \kappa_2(s)), \]
and
\[ \hat{\theta}_{\frac{\pi}{2}}^+ := \lim_{s \rightarrow s_0^+} 
\hat{\theta}\left(\frac{\sqrt{2}}{2}\kappa_1(s) +\frac{\sqrt{2}}{2}\kappa_2(s), 
-\frac{\sqrt{2}}{2}\kappa_1(s) +\frac{\sqrt{2}}{2} \kappa_2(s)\right). \]
Note that these limits may not exist.
Finally we define
\[ \theta^+ := 
\left\lbrace \begin{array}{cl} \hat{\theta}^+ & \mbox{if}\; -\frac{\pi}{2} < \hat{\theta}^+ < \frac{\pi}{2},  \\
\frac{\pi}{2} & \mbox{if}\;\;  \hat{\theta}_{\frac{\pi}{2}}^+ = -\frac{\pi}{4},\\
\mbox{undefined} & \mbox{otherwise}.
\end{array} \right. \]
Replacing $s \rightarrow s_0^+$ with $s \rightarrow s_0^-$ we similarly define $\theta^-$.
\begin{lem}\label{mainlem}
Let $(\kappa_1, \kappa_2)$ be $C^0$ with an isolated zero at $s_0$.  Then there exists a $C^0$ lift $(\tilde{r},\tilde{\theta})$ near $s_0$ such that $\pi \circ (\tilde{r},\tilde{\theta}) = (\kappa_1,\kappa_2)$ if and only if both $\theta^+$ and $\theta^-$ exist and $\theta^+= \theta^-$.
\end{lem}
\begin{proof}
First assume there does exist a $C^0$ lift $(\tilde{r},\tilde{\theta})$ near $s_0$ such that $\pi \circ (\tilde{r},\tilde{\theta}) = (\kappa_1,\kappa_2)$.  In other words for all $s$ near $s_0$ we have
\[\left(\tilde{r}(s) \cos \tilde{\theta}(s), \tilde{r}(s) \sin \tilde{\theta}(s)\right) = (\kappa_1(s),\kappa_2(s)).\]
Since the zero is isolated at $s_0$ we may assume $\tilde{r}(s) \neq 0$ near $s_0$.  

We next prove $\theta^+$ exists.  By definition this means that either $\hat{\theta}^+$ exists with $-\frac{\pi}{2} < \hat{\theta}^+ < \frac{\pi}{2}$ and/or $\hat{\theta}_{\frac{\pi}{2}}^+$ exists with $\hat{\theta}_{\frac{\pi}{2}}^+ = - \frac{\pi}{4}$.  By continuity
\begin{equation}\label{main}
\lim_{s \rightarrow s_0^+} \tilde{r}(s) = 0 = \lim_{s \rightarrow s_0^-} \tilde{r}(s)\;\;\; \mbox{and}\;\; \lim_{s \rightarrow s_0^+} \tilde{\theta}(s) = \tilde{\theta}(s_0) = \lim_{s \rightarrow s_0^-} \tilde{\theta}(s).
\end{equation}
When $s \neq s_0$ we have
\begin{eqnarray*}
\hat{\theta}(\kappa_1(s),\kappa_2(s)) & = & \left\lbrace \begin{array}{cc}
\tan^{-1}  \left( \frac{\tilde{r}(s) \sin \tilde{\theta}(s)}{\tilde{r}(s) \cos \tilde{\theta}(s)} \right)   &  \mbox{if}\; \cos \tilde{\theta}(s) \neq 0, \\ 
\frac{\pi}{2}& \mbox{if}\; \cos \tilde{\theta}(s) = 0.
\end{array}  \right. \\
\Longrightarrow \hat{\theta}(s) & = &  \left\lbrace \begin{array}{cc}
\tan^{-1}  \left( \tan (\tilde{\theta}(s))\right) &  \mbox{if}\; \tilde{\theta}(s) \bmod \pi \neq \frac{\pi}{2},  \\ 
\frac{\pi}{2}& \mbox{if}\; \tilde{\theta}(s) \bmod \pi = \frac{\pi}{2}.
\end{array}  \right.
\end{eqnarray*}
By conditions (\ref{main}) eventually $\tilde{\theta}(s)$ is near $\tilde{\theta}(s_0)$.  If $\tilde{\theta}(s_0) \bmod \pi \neq \frac{\pi}{2}$, then eventually $\tilde{\theta}(s) \bmod \pi \neq \frac{\pi}{2}$ and 
\[\theta^+ = \hat{\theta}^+ = \tan^{-1} \left(\tan (\tilde{\theta}(s_0))\right).\]
If $\tilde{\theta}(s_0) \bmod \pi = \frac{\pi}{2}$, then we rotate by $\frac{\pi}{4}$ and by the same argument we have 
\[\hat{\theta}_{\frac{\pi}{2}}^+ = \tan^{-1}\left(\tan(\tilde{\theta}(s_0)+\frac{\pi}{4})\right) = \tan^{-1}\left(\tan(\frac{3 \pi}{4})\right) = -\frac{\pi}{4}.\]
So $\theta^+ = \frac{\pi}{2}$.  Thus in either case $\theta^+$ exists.

Similarly $\theta^-$ exists and by condition (\ref{main}) $\theta^+ = \theta^-$.

Conversely assume $\theta^+$ and $\theta^-$ exist and $\theta^+=\theta^-$.  Let $(\tilde{r}^+(s),\tilde{\theta}^+(s))$ (resp.  $(\tilde{r}^-(s),\tilde{\theta}^-(s))$ be any $C^0$ lift for $s > s_0$ (resp. $s < s_0$).  Still assuming $r(s) \neq 0$ for $s \neq s_0$, without loss of generality assume both $\tilde{r}^+(s) > 0$ and $\tilde{r}^-(s) > 0$ for $s \neq s_0$.  

First we consider the Case 1 where $-\frac{\pi}{2} < \hat{\theta}^+ < \frac{\pi}{2}$ (and hence  $-\frac{\pi}{2} < \hat{\theta}^- < \frac{\pi}{2}$).  We want to show there is a $C^0$ lift $(\tilde{r},\tilde{\theta})$.  We claim $\lim_{s \rightarrow s_0^+} \tilde{\theta}^+(s)$ and $\lim_{s \rightarrow s_0^-} \tilde{\theta}^-(s)$ exist.  More precisely we have
\[ -\frac{\pi}{2} < \lim_{s \rightarrow s_0^+} \tan^{-1} \left(\frac{\kappa_2(s)}{\kappa_1(s)}\right) < \frac{\pi}{2}.\]
Or
\[ -\frac{\pi}{2} < \lim_{s \rightarrow s_0^+} \tan^{-1} \left(\frac{\tilde{r}^+(s) \sin \tilde{\theta}^+(s)}{\tilde{r}^+(s) \cos \tilde{\theta}^+(s)}\right) < \frac{\pi}{2}.\]
Or
\[ -\frac{\pi}{2} < \lim_{s \rightarrow s_0^+} \tan^{-1} \left(\tan \tilde{\theta}^+(s)\right) < \frac{\pi}{2}.\]
Eventually $\tilde{\theta}^+(s) \bmod \pi$ avoids $\frac{\pi}{2}$.  Thus eventually 
\begin{equation}\label{pie}\tan^{-1} \left(\tan \tilde{\theta}^+(s)\right) = 
\left\lbrace
\begin{array}{c} \tilde{\theta}^+(s) \bmod \pi \\ 
\mbox{or} \;\; \left( \tilde{\theta}^+(s) \bmod \pi \right) - \pi. \end{array} \right. 
\end{equation}
 In either case $\tilde{\theta}^+(s_0) = \lim_{s \rightarrow s_0^+} \tilde{\theta}^+(s)$ exists.  By the same argument  $\tilde{\theta}^-(s_0) = \lim_{s \rightarrow s_0^-} \tilde{\theta}^-(s)$ exists.  Since $\theta^+=\theta^-$ we know by Equation (\ref{pie}) that $\tilde{\theta}^+(s_0)-\tilde{\theta}^-(s_0)$ is a multiple of $\pi$.  Since we have assumed both $\tilde{r}^+(s) >0$ and $\tilde{r}^-(s)>0$ we can define $(\tilde{r}(s),\tilde{\theta}(s))$ depending on whether $(\kappa_1,\kappa_2)$ approaches the origin from the same or opposite directions  as $s$ approaches $s_0$ from the left and right.  In the first case we have $j(s_0)=\tilde{\theta}^+(s_0)-\tilde{\theta}^-(s_0) = 0 \bmod 2\pi$ and
 \[
 (\tilde{r}(s),\tilde{\theta}(s)) := 
 \left\lbrace 
 \begin{array}{cc}
 (\tilde{r}^-(s),\tilde{\theta}^-(s) + j(s_0))& \mbox{if}\;\; s < s_0,\\ 
 (0,\tilde{\theta}^+(s_0))& \mbox{if}\;\; s = s_0, \\
 (\tilde{r}^+(s),\tilde{\theta}^+(s))& \mbox{if}\;\; s > s_0.
 \end{array}
  \right. 
 \]
 If  $j(s_0)=\tilde{\theta}^+(s_0)-\tilde{\theta}^-(s_0) = \pi \bmod 2\pi$ then
 
 \[
 (\tilde{r}(s),\tilde{\theta}(s)) := 
 \left\lbrace 
 \begin{array}{cc}
 (-\tilde{r}^-(s),\tilde{\theta}^-(s) + j(s_0))& \mbox{if}\;\; s < s_0,\\ 
 (0,\tilde{\theta}^+(s_0))& \mbox{if}\;\; s = s_0, \\
 (\tilde{r}^+(s),\tilde{\theta}^+(s))& \mbox{if}\;\; s > s_0.
 \end{array}
  \right. 
 \]

In Case 2 we assume $\hat{\theta}^+=\frac{\pi}{2}$.  Since this implies $\hat{\theta}_{\frac{\pi}{2}}^+ = -\frac{\pi}{4}$ we see that Case 2 is included in Case 3.

Finally we consider Case 3:
\[\hat{\theta}_{\frac{\pi}{2}}^+ = -\frac{\pi}{4}.\]
As discussed above $(\kappa_1,\kappa_2)$ is oscillating across the $y$-axis, but otherwise converges nicely.  We can $C^0$ lift the rotated $(\kappa_1,\kappa_2)$ and then shift that $(\tilde{r},\tilde{\theta})$ by $\frac{\pi}{4}$.
\end{proof}
As mentioned above, if the conditions of Lemma \ref{mainlem} are valid at all points of zero curvature, then we have a global $C^0$ lift $(\tilde{r},\tilde{\theta})$ of $(\kappa_1,\kappa_2)$.  Without loss of generality we will assume $\tilde{\theta}(0)=0$.

\section{The Beta Frame}
\subsection{Initial Conditions}
Without loss of generality we assume $\kappa^f(0) \neq 0$ and define $M_1(s), M_2(s)$ by the following initial conditions:
\begin{enumerate}
\item $\mbox{If $\gamma$ is planar, then } M_1(0)=N(0), M_2(0)=T(0) \times N(0)$,
\item $\mbox{If $\gamma$ is not-planar, then } M_1(0)=N^f(0), M_2(0)=B^f(0)$.
\end{enumerate}
\subsection{The Beta Normal, Signed Curvature and Binormal}
Assume that the normal development of $\gamma$ has a continuous lift $(\tilde{r},\tilde{\theta})$ with $\tilde{\theta}(0)=0$ as in Section \ref{lift} so that 
\[(\kappa_1,\kappa_2) = (\tilde{r} \cos \tilde{\theta},\tilde{r} \sin\tilde{\theta}). \]
Then we can globally define $N^\beta$ by
\[N^\beta(s) := \cos{\tilde{\theta}(s)}\;M_1(s) + \sin{\tilde{\theta}(s)} \; M_2(s)\]
and note that
\[N^\beta = \pm N^f\]
whenever $N^f$ is defined.
We define our signed curvature $\kappa^\beta$ by 
\[\dot{T}(s) =: \kappa^\beta(s) N^\beta(s).\]
Note
\[\kappa^\beta = \tilde{r}\]
and 
\[\kappa^\beta = \pm \kappa^f\]
whenever $\kappa^f$ is defined.
We define $B^\beta$ by 
\[B^\beta(s) := T(s) \times N^\beta(s)\]
and note that
\[B^\beta = -\sin \tilde{\theta}\; M_1 + \cos \tilde{\theta} \;M_2.\]
The globally defined frame $\{T,N^\beta,B^\beta \}$ is called the Beta frame.  The Beta frame (when defined) is unique and is invariant under regular, orientation preserving, base point fixing reparametrizations.  If the base point changes, it may happen that $N^\beta$ and $B^\beta$ globally switch signs.

\section{Torsion}
Finally we assume $\gamma$ is $C^3$, $\Vert \dot{\gamma}\Vert \equiv 1$ and assume there exists a continuous lift $(\tilde{r},\tilde{\theta})$ of $(\kappa_1,\kappa_2)$ as in the last section.  In this case we have that $\kappa_1$, $\kappa_2$ are $C^1$ and after a bit of checking $\tilde{r}$ is $C^1$.  $\tilde{\theta}$ is once again more difficult.  Even if both $\kappa_1$, $\kappa_2$ are $C^1$, there is no guarantee that $\tilde{\theta}$ is $C^1$ if $\kappa^\beta(s_0)=0$ (even if $\Vert (\dot{\kappa_1}(s_0),\dot{\kappa_2}(s_0))\Vert \neq 0$.  For example $(\tilde{r},\tilde{\theta}) = (s, s^\frac{1}{3})$.  See Figure \ref{c0notc1}.

\begin{figure}[h!]
\begin{center}
\includegraphics[scale=.6]{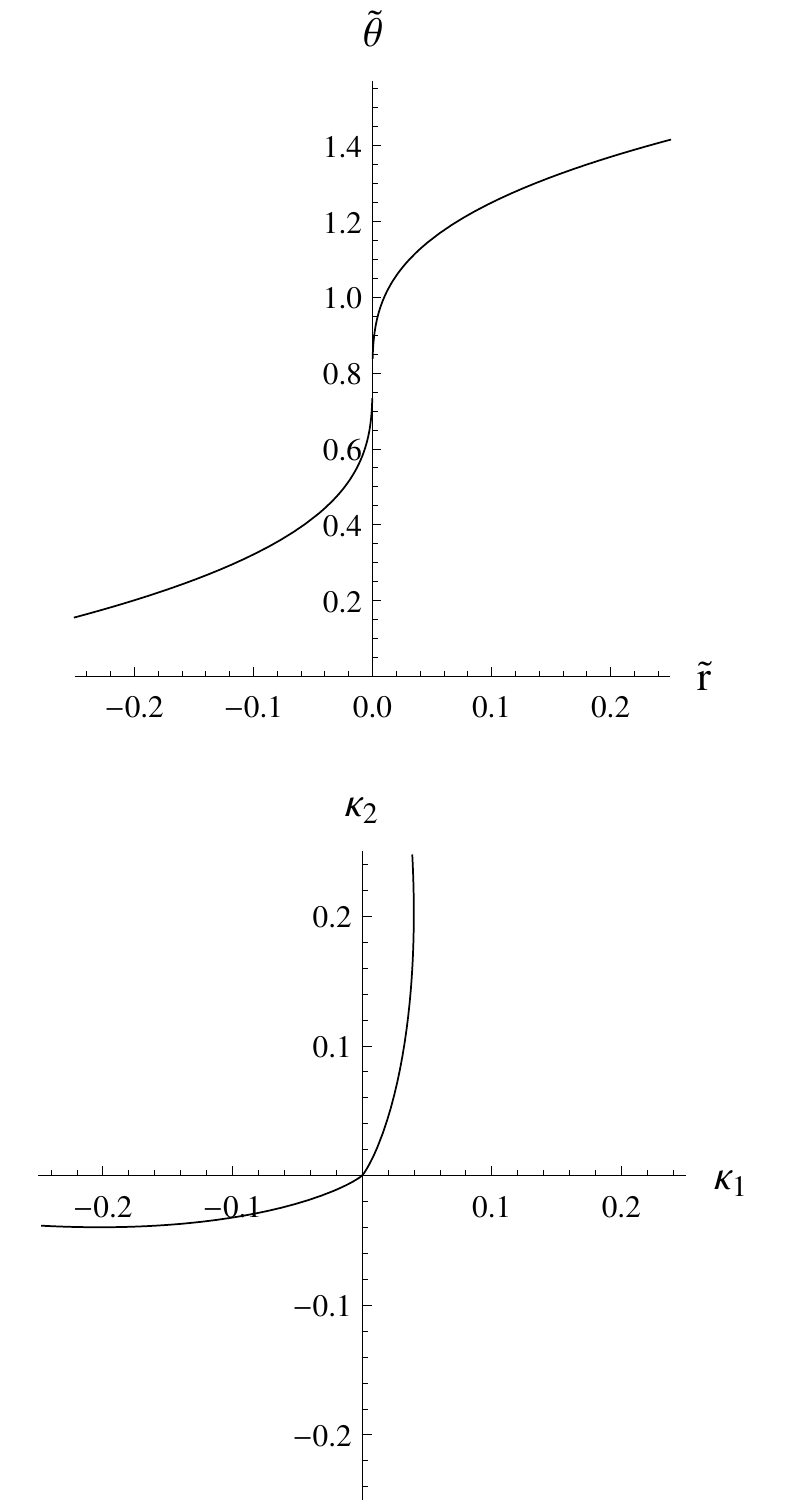}
\caption{$(\tilde{r},\tilde{\theta})=(s,s^\frac{1}{3})$ An Example Where $\tilde{\theta}$ is $C^0$ but not $C^1$}\label{c0notc1}
\end{center}
\end{figure}

Assuming $\tilde{\theta}(s)$ is $C^1$ at all curvature zero points, then the lift $(\tilde{r},\tilde{\theta})$ is globally $C^1$ and we define $\tau^\beta$ by 
\[ \tau^\beta(s) := \dot{\tilde{\theta}}(s)\]
and note that
\[\tau^\beta = \tau^f\]
whenever $\tau^f$ is defined.

As promised we will still have the Frenet equations:
\begin{alignat*}{4}
\dot{T} &=& \kappa^\beta N^\beta, \\
\dot{N^\beta} &=-\kappa^\beta T &&+\tau^\beta B^\beta, \\
\dot{B^\beta} &=&-\tau^\beta N^\beta.
\end{alignat*}

\end{document}